\documentclass[12pt]{amsart}

\usepackage{amsmath}
\usepackage{amssymb}
\usepackage[curve]{xypic}
\usepackage{caption}
\usepackage{xspace}
\usepackage{cite}
\usepackage{enumitem}
\usepackage{url}
\usepackage{textcomp}
\usepackage{tikz}
\usepackage{color}
\usepackage{xcolor}
\usetikzlibrary{patterns}
\usepackage[section]{placeins}

\makeatletter 
\def\@cite#1#2{{\m@th\upshape\bfseries%
[{#1\if@tempswa{\m@th\upshape\mdseries, #2}\fi}]}}
\makeatother 

\theoremstyle{plain}
\newtheorem{thm}{Theorem}[section]
\newtheorem{cor}[thm]{Corollary}
\newtheorem{prop}[thm]{Proposition}
\newtheorem{lem}[thm]{Lemma}

\theoremstyle{definition}
\newtheorem*{notation}{Notation}
\newtheorem{defn}[thm]{Definition}

\theoremstyle{remark}
\newtheorem{rem}[thm]{Remark}

\numberwithin{equation}{subsection}
\renewcommand{\bold}[1]{\medskip \noindent {\bf #1 }\nopagebreak}

\newcommand{\nc}{\newcommand}
\newcommand{\rnc}{\renewcommand}

\newcommand{\bk}{{\mathbf{k}}}

\nc\bA{\mathbb{A}}
\nc\bB{\mathbb{B}}
\nc\bC{\mathbb{C}}
\nc\bD{\mathbb{D}}
\nc\bE{\mathbb{E}}
\nc\bF{\mathbb{F}}
\nc\bG{\mathbb{G}}
\nc\bH{\mathbb{H}}
\nc\bI{\mathbb{I}}
\nc{\bJ}{\mathbb{J}}
\nc\bK{\mathbb{K}}
\nc\bL{\mathbb{L}}
\nc\bM{\mathbb{M}}
\nc\bN{\mathbb{N}}
\nc\bO{\mathbb{O}}
\nc\bP{\mathbb{P}}
\nc\bQ{\mathbb{Q}}
\nc\bR{\mathbb{R}}
\nc\bS{\mathbb{S}}
\nc\bT{\mathbb{T}}
\nc\bU{\mathbb{U}}
\nc\bV{\mathbb{V}}
\nc\bW{\mathbb{W}}
\nc\bY{\mathbb{Y}}
\nc\bX{\mathbb{X}}
\nc\bZ{\mathbb{Z}}
\nc\cA{\mathcal{A}}
\nc\cB{\mathcal{B}}
\nc\cC{\mathcal{C}}
\rnc\cD{\mathcal{D}}
\nc\cE{\mathcal{E}}
\nc\cF{\mathcal{F}}
\nc\cG{\mathcal{G}}
\rnc\cH{\mathcal{H}}
\nc\cI{\mathcal{I}}
\nc{\cJ}{\mathcal{J}}
\nc\cK{\mathcal{K}}
\rnc\cL{\mathcal{L}}
\nc\cM{\mathcal{M}}
\nc\cN{\mathcal{N}}
\nc\cO{\mathcal{O}}
\nc\cP{\mathcal{P}}
\nc\cQ{\mathcal{Q}}
\rnc\cR{\mathcal{R}}
\nc\cS{\mathcal{S}}
\nc\cT{\mathcal{T}}
\nc\cU{\mathcal{U}}
\nc\cV{\mathcal{V}}
\nc\cW{\mathcal{W}}
\nc\cY{\mathcal{Y}}
\nc\cX{\mathcal{X}}
\nc\cZ{\mathcal{Z}}

\nc{\dmo}{\DeclareMathOperator}
\dmo{\Tw}{Twist}
\dmo{\CP}{Pres}
\rnc{\Re}{\operatorname{Re}}
\rnc{\Im}{\operatorname{Im}}
\rnc{\span}{\operatorname{span}}
\dmo{\rank}{rank}
\dmo{\End}{End}
\dmo{\Ad}{Ad}
\dmo{\Id}{Id}
\dmo{\lcm}{lcm}
\dmo{\Area}{Area}

\nc{\Tm}{Teichm\"uller\xspace}
\nc{\odd}{\cH^{\text{odd}}(4)}
\nc{\hyp}{\cH^{\text{hyp}}(4)}
\nc{\prym}{\tilde{\mathcal{Q}}(3,-1^3)}
\nc{\G}{GL^+(2,\bR)}

\begin{document}

\title[Finiteness in non-arithmetic rank 1]{Finiteness of Teichm\"uller curves in non-arithmetic rank 1 orbit closures}
\author[Lanneau]{Erwan~Lanneau}
\address{\hspace{-0.5cm} Institut Fourier, Universit\'e de Grenoble 1 \newline
100 rue des Maths, BP 74 \newline
38402 Saint-Martin d'H\`eres \newline FRANCE}
\email{Erwan.Lanneau@ujf-grenoble.fr}
\author[Nguyen]{Duc-Manh~Nguyen}
\address{\hspace{-0.5cm} IMB Bordeaux-Universit\'e de Bordeaux \newline
351, Cours de la Lib\'eration \newline
33405 Talence Cedex \newline FRANCE}
\email{duc-manh.nguyen@math.u-bordeaux1.fr}
\author[Wright]{Alex~Wright}
\address{\hspace{-0.5cm} Department of Mathematics\newline
Stanford University\newline
Palo Alto, CA 94305}
\email{amwright@stanford.edu}
%


\begin{abstract}
We show that in any non-arithmetic rank 1 orbit closure of translation surfaces, there are only finitely many Teichm\"uller curves. We also show that in any non-arithmetic rank 1 orbit closure, any completely parabolic surface is  Veech.
\end{abstract}

\maketitle
\thispagestyle{empty}


\section{Introduction}\label{S:intro}

\subsection{Motivation and  results}
Closed $GL(2, \bR)$ orbits of translation surfaces first appeared in the work of Veech in connection to billiards in triangles \cite{V}. Now their study is a sizable industry combining flat geometry, dynamics, Teichm\"uller theory, algebraic geometry, and number theory. Examples arising from torus covers are abundant, and additionally a few infinite families and  sporadic examples have been discovered, but following deep work of \mbox{McMullen}, M\"oller, Bainbridge and others it is  believed that most types of closed $GL(2, \bR)$ orbits are rare: there should be at most finitely many in each genus.

Recent work of  Matheus-Wright demonstrates a new paradigm for proving such finiteness results \cite{MW}. The key step is to show that closed $GL(2, \bR)$ orbits of the given type are not dense in some larger orbit closure $\cM$. This $\cM$ might be either \emph{higher rank} or  \emph{rank 1}.
Here we rule out the rank 1 case.
\par
\begin{thm}\label{T:fin}
Each non-arithmetic rank 1 orbit closure contains at most finitely many closed $GL(2, \bR)$ orbits.
\end{thm}
The known rank 1 orbit closures for which Theorem \ref{T:fin} is new are the Prym eigenform loci in genus 4 and 5 and the Prym eigenform loci in genus 3 in the principal stratum.

A point on a closed $GL(2, \bR)$ orbit is called a Veech surface. Many strange and fascinating surfaces are known that share some properties with Veech surfaces but are not actually Veech. Most of these lie in rank 1 orbit closures. An open question of Smillie and Weiss asks whether a particular property of Veech surfaces--complete parabolicity--in fact implies that the surface is Veech \cite{SWprobs}. We give a positive answer in rank 1 and in genus 2.

\begin{thm}\label{T:CPveech}
In a rank 1 orbit closure, every completely parabolic surface is Veech.
\end{thm}

\begin{cor}
In genus 2 and in the Prym loci,  every completely parabolic surface is Veech.
\end{cor}

The Prym loci, which exist in genus 3, 4, and 5, are defined in \cite{Mc2}. The corollary follows from Theorem \ref{T:CPveech} because in genus 2 and in the Prym loci every completely parabolic surface lies in a rank 1 orbit closure \cite{Mc, Mc2}. 

Theorem \ref{T:CPveech} is new even for the much studied eigenform loci in genus 2, and sheds light on the extent to which orbit closures can be computed using cylinder deformations \cite{Wcyl}.

\subsection{Terminology}

Veech surfaces satisfy a remarkable property: in any direction, the straight line flow in that direction is either uniquely ergodic or there is a collection of parallel cylinders in that direction that cover the surface and whose moduli  are all rational multiples of each other \cite{V}. (The modulus of a cylinder is defined to be the height divided by the circumference). This motivates the following definition.
%
\par
\begin{defn}
A translation surface is called \emph{completely parabolic} if for any cylinder direction, there is a collection of parallel  cylinders in that direction that cover the surface and whose moduli  are all rational multiples of each other.
\end{defn}
\par
An orbit closure of  translation surfaces is a closed orbit if and only if it contains only completely parabolic surfaces \cite{V}. Relaxing this restriction, one arrives at one definition of a rank 1 orbit closure.

\reversemarginpar
\begin{defn}
A translation surface is called \emph{completely periodic} if for every cylinder direction there is a collection of parallel cylinders in that direction that cover the surface. An orbit closure of translation surfaces is called \emph{rank 1} if it consists entirely of completely periodic translation surfaces. Closed $GL(2,\bR)$ orbits, or more generally rank 1 orbit closures, are called \emph{arithmetic} if they consist entirely of branched torus covers \cite{Wfield}.
\end{defn}

\par
See Section~\ref{sec:rank} and Theorem~\ref{thm:W} for equivalent definitions.  In general the torus covers can be branched over several points, however for closed $GL(2,\bR)$ orbits they can be assumed to be branched only over one point.
\par

\subsection{Outline of proof} Both Theorems \ref{T:fin} and \ref{T:CPveech} will be deduced from the following result.
\par
\begin{thm}\label{T:mother}
The set of completely parabolic surfaces is not dense in any non-arithmetic rank 1 affine invariant
submanifold that is not a closed orbit.
\end{thm}

Theorem \ref{T:mother} will be established in Section~\ref{S:proof} via a mostly elementary argument, showing that in non-arithmetic rank 1 affine invariant
submanifold most rel deformations of a surface with sufficiently large Veech group are not completely parabolic. 
The key tools will be the Zariski density of the Veech group in $\mathrm{SL}(2,\bR)^d$ and a result of Eskin-Mozes-Oh on escape from sub-varieties
in Zariski dense groups.

To derive Theorems \ref{T:fin} and \ref{T:CPveech} from Theorem \ref{T:mother}, we will use the recent work of Eskin-Mirzakhani-Mohammadi \cite{EM, EMM}.

\begin{thm}[Eskin-Mirzakhani-Mohammadi]\label{T:EMM}
Any closed $\mathrm{GL}(2, \bR)$ invariant subset of a stratum is a finite union of affine invariant submanifolds.
\end{thm}

\begin{proof}[Proof of Theorems~\ref{T:fin} and~\ref{T:CPveech} from Theorem~\ref{T:mother}]
We first prove Theorem~\ref{T:fin}. Let us assume that there are infinitely many closed orbits $\cC_i$ in a non-arithmetic rank 1 orbit closure $\cN$. By Theorem \ref{T:EMM}, the closure of the union of the $\cC_i$ is equal to a finite union of affine invariant submanifolds $\cM_1\cup \cdots\cup \cM_k$. One of these orbit closures, say $\cM_1$, must contain infinitely many of the $\cC_i$. Since $\cN$ is rank 1 and $\cM_1\subset \cN$, it follows that $\cM_1$ is also rank 1. 
 Since
$\cN$ is non-arithmetic,  $\cM_1$ is also non-arithmetic (Corollary~\ref{C:same}).
Because closed orbits consist entirely of completely parabolic translation surfaces, this contradicts Theorem \ref{T:mother}.

We now prove Theorem~\ref{T:CPveech}. Suppose $\cN$ is a non-arithmetic rank 1 orbit closure, and suppose $(X, \omega)\in \cN$ is a  completely parabolic surface that is not Veech.
Let $\cM$ be the orbit closure of $(X, \omega)$. Again we see that $\cM$ must be non-arithmetic and rank 1. Because $(X, \omega)$ is not Veech, $\cM$ is not a closed orbit. Furthermore $\cM$ contains a dense set of completely parabolic surfaces given by the orbit of $(X, \omega)$. This contradicts Theorem \ref{T:mother}.
\end{proof}

\subsection{Context}\label{sec:context} A Teichm\"uller curve is an isometrically immersed curve in the moduli space of Riemann surfaces. The projection of a closed $GL(2, \bR)$ orbit to the moduli space of Riemann surfaces is a Teichm\"uller curve, so as is common we will also refer to closed $GL(2, \bR)$ orbits as Teichm\"uller curves. 

\bold{Marked points on Veech surfaces.} The easiest examples of rank 1 orbit closures that are not Teichm\"uller curves are given by orbit closures of Veech surfaces with finitely many marked points. One can also take covers of these surfaces branched over the marked points and obtain very closely related rank 1 orbit closures of translation surfaces with no marked points.

A Veech surface with a marked point has closed $GL(2, \bR)$ orbit if and only if the marked point is periodic, i.e., has finite orbit under the action of the affine group. Gutkin-Hubert-Schmidt \cite{GHS} and M\"oller \cite{M2} showed that on any non-arithmetic Veech surface there are at most finitely many periodic points. This can be viewed as a special case of Theorem \ref{T:fin}, and was the inspiration for the present work. More specifically, the paper of Gutkin-Hubert-Schmidt  motivated this paper by giving us hope that in rank 1 a general finiteness result could be obtained without using much algebraic geometry.


\bold{Rank 1 orbit closures.} There are additionally infinitely many rank 1 orbit closures in genus 2 through 5, called the (Prym) eigenform loci. These examples were discovered by McMullen; in the case of genus 2 they were independently discovered in a different guise by Calta \cite{Ca, Mc, Mc2}. All currently known rank 1 orbit closures arise from (Prym) eigenform loci or Teichm\"uller curves via branched covering constructions.


Arithmetic rank 1 orbit closures never contain non-arithmetic Teichm\"uller curves, and they always contain a dense set of arithmetic Teichm\"uller curves. Non-arithmetic orbit closures never contain arithmetic Teichm\"uller curves \cite{Wfield}.

\bold{Unexpected surfaces.} The (Prym) eigenform loci and branched covers of Veech surfaces have been a rich source of examples of  surfaces with peculiar properties, including surfaces with infinitely generated Veech groups and satisfying the (topological) Veech dichotomy \cite{LN, Mc6, SW3, HS2}.

There are some examples of unusual surfaces not lying in a rank 1 orbit closure, for example there are many  completely periodic but not Veech surfaces in $\cH(4)^{\rm{hyp}}$ \cite{LN}, whose orbits must be dense by \cite{NW}.

\bold{Teichm\"uller curves.} A Teichm\"uller curve is called primitive if it does not arise from a covering construction \cite{M2}. 
Primitive Teichm\"uller curves were classified in genus 2 by McMullen \cite{McM:spin, Mc4}. 
As part of this classification,
McMullen showed that there is only one primitive Teichm\"uller curve in $\cH(1,1)$ \cite{Mc4}. Before proving this, McMullen established the weaker result that in each non-arithmetic eigenform locus in genus 2 there are at most finitely many Teichm\"uller curves \cite{Mc7}. Theorem \ref{T:fin} is a generalization of this result.

Furthermore, McMullen was able to show that every non-Veech surface in these loci must have a saddle connection not parallel to any cylinders \cite[Theorem 7.5]{Mc7}. This is very much analogous to Theorem \ref{T:CPveech}, because both imply the existence of a flat geometry ``certificate" that shows that a surface is not Veech. Our techniques also have similarities with those of McMullen. For example, both use that an infinite set of Teichm\"uller curves must be dense. (McMullen's work was well before Theorem \ref{T:EMM} was established, but he had previously established this result in genus 2 \cite{Mc5}.)

A Teichm\"uller curve is called algebraically primitive if the trace field of its Veech group has degree equal to the genus. Finiteness of algebraically primitive Teichm\"uller curves was shown in $\cH(g-1,g-1)^{hyp}$ by M\"oller \cite{M3}, in $\cH(3,1)$ by Bainbridge-M\"oller \cite{BaM}, in general in genus 3 by Bainbridge-Habegger-M\"oller \cite{BHM}, and in the minimal stratum in prime genus at least 3  by Matheus-Wright \cite{MW}. In the minimal stratum in genus 3, there are only finitely many primitive Teichm\"uller curves not contained in the Prym locus \cite{MW, NW, ANW}.

There is a partial classification of primitive Teichm\"uller curves in the Prym eigenform loci in genus three~\cite{ML}. Before the work of Bainbridge-Habegger-M\"oller, Lanneau-Nguyen had previously established Theorem~\ref{T:fin} for many Prym eigenform loci in genus three  in genus 3~\cite{LN}, using the strategy that McMullen used in eigenform loci in genus 2 \cite{Mc7}.

Teichm\"uller curves have extra-ordinary algebro-geometric properties \cite{M, M2, Mc3}, which have been key tools in much previous work on  finiteness of Teichm\"uller curves. Our method will not use any of these algebro-geometric results.

\bold{Acknowledgments.} The authors thank Emmanuel Breuillard, Yves Benoist and Slavyana Geninska for helpful conversations. This collaboration began at the program ``Flat Surfaces and Dynamics on Moduli Space" at Oberwolfach in March 2014, and continued at the program ``Introductory Workshop: Geometric and Arithmetic Aspects of Homogeneous Dynamics" at MSRI in February 2015. All three authors attended both programs and are grateful to the organizers, the MFO, and MSRI. The third author is partially supported by a Clay Research Fellowship. The first author is partially supported by the ANR Project GeoDyM. The authors thank Curtis McMullen and Ronen Mukamel for helpful comments on an earlier draft.

\section{Background}\label{S:background}

We review useful results concerning affine invariant submanifolds, rank and (affine) field of
definition. For a general introduction to translation surfaces and their moduli spaces,
we refer the reader to the surveys~\cite{MT,Wsurvey,Z}, or any of the many other surveys listed in \cite{Wsurvey}.
\par

Let $\cT_g$ be the Teichm\"uller space of genus $g$ Riemann surfaces. Over the complex manifold $\cT_g$ there is a natural bundle $\cT_{g,1}\to \cT_g$ whose fiber over the point $X\in \cT_g$ is the Riemann surface $X$. This bundle admits a natural action of the mapping class group $\Gamma$ covering the usual action of $\Gamma$ on $\cT_g$. The action of $\Gamma$ on $\cT_g$  has fixed points, so not all fibers of the map $\cT_{g,1}/\Gamma \to \cT_g/\Gamma$ are surfaces of genus $g$. 

One solution to this issue is to use stacks. Here we will instead use a finite index subgroup $\Gamma'\subset \Gamma$ whose action on $\cT_g$ does not have fixed points.
One standard choice of $\Gamma'$ is the level 3 congruence subgroup, i.e. the kernel of the action of the mapping class group on the first homology of a surface with $\bZ/3$ coefficients, see for example \cite[Theorem 6.9]{FarMar}.  In this case $\cT_g/\Gamma'$ parameterizes genus $g$ Riemann surfaces with a choice of level 3 structure, i.e. a choice of basis for first homology with $\bZ/3$ coefficients, and $\cT_{g,1}/\Gamma' \to \cT_g/\Gamma'$ is a fiber bundle whose fibers are genus $g$ Riemann surfaces.

Throughout this paper,  $\cH_g$ will denote the moduli space of Abelian differentials $\omega$ on Riemann surfaces $X$ of genus $g\geq 1$ equipped with a choice of level three structure. For any non-negative integer partition $(k_1, k_2, \ldots, k_s)$ of $2g-2$ we will denote by $\cH(k_1, \dots, k_s)\subset \cH_g$ the stratum of translation surfaces having $s$ labelled zeros, of orders $k_1, \ldots, k_s$. We will allow translation surface to have finitely many marked points, and by convention we will consider a marked point  to be a zero of order 0.
The marked points are required to be distinct from each other and the set of true zeros.
$\mathrm{GL}(2,\mathbb{R})$ acts on the set of translation surfaces (by post-composition with the charts of translation atlas). Moreover, this $GL(2,\mathbb{R})$-action on Abelian differential preserves each stratum. We will denote the stabilizer of $(X,\omega)$ by $\mathrm{SL}(X,\omega)$, and we will refer to it as the {\em Veech group}.
\par
The level three structure may be safely ignored, and we will drop it from our notation. (The use of level structures does affect  $SL(X,\omega)$ up to finite index, as does the labeling of the zeros. However this is not relevant in the present paper.) We will write $(X,\omega)\in \cH$, and it will be implicit that $X$ carries a level three structure and the requisite number of labelled marked points.

\subsection{Affine invariant submanifolds}

Let $\Sigma$ be the set of zeros of $(X,\omega)\in \cH$ and let $\gamma_1, \ldots, \gamma_n$ be any basis of the relative homology group $H_1(X, \Sigma, \bZ)$,
where $n=2g+s-1$. The period coordinate maps defined by
$$
(X,\omega)\mapsto\left( \int_{\gamma_1} \omega, \ldots, \int_{\gamma_n}\omega \right)
$$
provide $\cH$ with an atlas of charts to $\bC^n$ with transition functions in $GL(n, \bZ)$ (see~\cite{K}). If $\gamma$ is in absolute homology $H_1(X, \bZ)$, the integral $\int_\gamma \omega$ is called an absolute period, and if $\gamma$ is merely in relative homology $H_1(X, \Sigma, \bZ)$ it is called a relative period.
\par
Over each stratum there are flat bundles $H^1_{rel}$ and $H^1$ whose fibers over $(X,\omega)$ are $H^1(X, \Sigma, \bC)$ and $H^1(X, \bC)$ respectively. There is a natural map $p:H^1_{rel}\to H^1$.
\reversemarginpar

\par
\begin{defn} An \emph{affine invariant submanifold} of $\cH$ is  a properly immersed manifold $\cM\hookrightarrow \cH$ such that each point of $\cM$ has a neighborhood whose image is given  by the set of surfaces in an open set satisfying a set of real linear equations in  period coordinates.
\par
These linear equations define a flat subbundle $T(\cM)$ of the pull back of $H^1_{rel}$ to $\cM$. The fiber $T_x (\cM)$ at a point $x\in \cM$ is given by the space of all relative cohomology classes that satisfy all the linear equations which define the image of a neighborhood of $x$ in period coordinates.
\end{defn}
It follows directly from the definition that $\cM$ is $GL(2,\bR)$-invariant.
For notational simplicity, we will treat affine invariant submanifold as subsets of strata. However technically all arguments should be phrased in the  manifold $\cM$ rather than the stratum, since for example the tangent space cannot be defined at points of self-crossing of the image of $\cM$ in the stratum. A point in $\cM$ can be considered to be a point in $(X,\omega)$ in the image of $\cM$ together with a subspace of $H^1(X, \Sigma, \bC)$. The subspace is the tangent space $T_{(X,\omega)}(\cM)$, and is uniquely determined by $(X,\omega)$ unless $(X,\omega)$ is in the self-crossing locus of $\cM$, in which case there may be finitely many choices. The self-crossing locus consists of a finite union of smaller dimensional affine invariant submanifolds.

\par
\begin{defn}
A \emph{rel deformation} is a path of translation surfaces $\{(X_t, \omega_t)\}_{t\in [0,1]}$ along which all absolute periods $\int_\gamma \omega_t$ stay constant.
\end{defn}

We define a path to be a continuous map of a closed interval into a space.

\begin{lem}
The bundle $\ker(p)$ is a trivial flat bundle over $\cH$. 
\end{lem}

%
\begin{proof}
Suppose the zeros of $\omega$ are $p_1, \ldots, p_s$.  Since the zeros are labelled, the ordering of the $p_i$ can be chosen consistently over $\cH$. At any point $(X, \omega)\in \cH$, for each $i=1, \ldots, s-1$, let $\gamma_i$ be a path on $X$ from $p_i$ to $p_{i+1}$. The map from the fiber of $\ker(p)$ to $\bC^{s-1}$ given by
$$v\mapsto \left( v(\gamma_i)\right)_{i=1}^{s-1}$$
does not depend on the choice of $\gamma_i$ and provides an isomorphism between $\ker(p)$ and the trivial bundle $\bC^{s-1}$ over $\cH$.

The $\gamma_i$ cannot be chosen to be global flat sections, but this is not relevant because the map does not depend on the choice of $\gamma_i$. The reason it does not depend on the choice of $\gamma_i$ is because different $\gamma_i$ differ by absolute homology classes, and by assumption the integral of $v$ over any absolute homology class is zero.
\end{proof}


\begin{notation}
Because $H^1(X,\Sigma, \bC)$ provides local coordinates for $\cH$ at $(X,\omega)$, for any sufficiently small $v\in H^1(X,\Sigma, \bC)$ there is a nearby point in $\cH$ whose period coordinates are given by those of $(X,\omega)$ plus $v$. We will denote this surface by $(X,\omega)+v$.
\end{notation}
\par
\begin{lem}\label{L:gonrel}
Fix $(X,\omega) \in \cH$. There is a neighborhood $U$ of $0$ in $\ker(p)$ such that on $U$ the map $v\mapsto (X,\omega)+v$ is well defined and injective, and such that if $v\in U$ then $tv\in U$ for all $t\in [0,1]$. Moreover, if $g\in SL(X,\omega)$ and $v, gv\in U$ then $g((X,\omega)+v)=(X,\omega)+gv$. 
\end{lem}
\par
 Here $g$ acts on $v\in \ker(p)$ by acting on the real and imaginary parts. That is, $\ker(p)=\bC^{s-1}$, and the action on coordinates is  the standard linear action of $GL(2, \bR)$ on the plane $\bC\simeq \bR^2$.
\begin{proof}
The first claim follows from the fact that  period coordinates are indeed local coordinates.

For the final claim, fix $g\in GL(2, \bR)$ and set $(X_v, \omega_v)=(X,\omega)+v$ and $(X_v^g, \omega_v^g)=g((X,\omega)+v).$

Now, if $\gamma$ is any relative cycle,
\begin{eqnarray*}
\int_\gamma \omega_v^g&=& g \cdot \int_\gamma \omega_v
\\&=& g\cdot \left(\int_\gamma \omega + v(\gamma)\right)
\\&=&g\cdot \int_\gamma \omega + g\cdot v(\gamma).
\end{eqnarray*}
This proves the result, using the triviality of $\ker(p)$.

%
 \end{proof}

\subsection{Rank of an affine invariant submanifold} \label{sec:rank}
We define the {\em rank} of an affine invariant submanifold $\cM$ by $\frac{1}{2}\ \mathrm{dim}_\bC p(T_{(X,\omega)}(\cM))$ for any $(X,\omega)\in \cM$.
(Note that by the work of Avila-Eskin-M\"oller~\cite{AEM}, $p(T_{(X,\omega)}(\cM))$ is symplectic, and thus has even dimension over $\bC$).
\par
\begin{thm}\label{thm:W}
An affine invariant submanifold $\cM$ consists entirely of completely periodic surfaces if and only if $\frac{1}{2} \mathrm{dim}_\bC p(T_{(X,\omega)}(\cM))$.
\end{thm}

That  $\frac{1}{2} \mathrm{dim}_\bC p(T_{(X,\omega)}(\cM))=1$ implies all surfaces in $\cM$ are completely periodic  was established in general in \cite{Wcyl}, having previously been known in a number of special cases \cite{LN, Ca, Mc7, V}. The simplest known proof, which in fact gives a stronger result, may be found in \cite{Wsurvey}. See \cite{Wcyl} for the converse.

\par
When $\cM$ is a rank 1 submanifold, $p(T_{(X,\omega)}(\cM))$ is spanned by the real and imaginary parts of the absolute cohomology class of $\omega$. Hence $\mathrm{span}(\mathrm{Re}(\omega),\mathrm{Im}(\omega))$ is locally constant, because  $p(T(\cM))$ is a flat subbundle. One thinks of $\mathrm{span}(\mathrm{Re}(\omega),\mathrm{Im}(\omega))$ as the directions coming from  $GL(2, \bR)$, because this subspace is always contained in $T(\cM)$ as a consequence of $GL(2, \bR)$ invariance.
The following basic result underlies all of our analysis.

\begin{prop}[\cite{Wsurvey}]\label{P:SL2plusrel}
An affine invariant submanifold $\cM$ is rank 1 if and only if for each $(X, \omega)\in \cM$ there is a neighborhood $U$ of the identity in $GL(2, \bR)$ and a neighborhood $V$ of $0$ in $\ker(p)\cap T_{(X,\omega)}(\cM)$ such that the map from $U\times V \to \cM$ given by $(g, v)\mapsto g(X,\omega)+v$ is a diffeomorphism onto a neighborhood of $(X, \omega)$ in $\cM$.
\end{prop}

\subsection{Stable directions}
Let $\cM$ be an affine invariant submanifold of rank one. We say that a surface $(X,\omega) \in \cM$ is {\em $\cM$-stably periodic} in some direction  if the surface  is periodic in this direction and if all saddle connections in this direction remain parallel on all sufficiently nearby surfaces in $\cM$. A periodic direction that is not $\cM$-stable will be called {\em $\cM$-unstable}.

\begin{rem}
If $(X,\omega)$ is horizontally $\cM$-stably periodic then there is an open subset $U$ of $\ker(p)\cap T_{(X,\omega)}(\cM)$ containing $0$ such that for any $v\in U$, $(M,\omega)+v$ is also horizontally $\cM$-stably periodic with the same cylinder diagram (compare to~\cite{LN2}).
\end{rem}

\begin{rem}
If $\cM$ is of higher rank, then no surface in $\cM$ admits a $\cM$-stable periodic direction \cite{Wcyl}.
\end{rem}

\begin{rem}
$\cM$-stability can be described in the language of \cite{Wcyl} by saying that the twist space is equal to the cylinder preserving space.
\end{rem}

\begin{rem}If $\cM$ is rank 1 and $(X,\omega) \in \cM$ is horizontally periodic and $U$ is a neighborhood of $0$ in $\ker(p)\cap T_{(X,\omega)}(\cM)$ on which the map $v \mapsto (X,\omega)+v$ is well defined and injective, then there is an open dense subset $V\subset U$ such that for any $v\in V$, $(X,\omega)+v$ is horizontally $\cM$-stably periodic. This is because the condition of being  $\cM$-unstable periodic gives rise to  linear equations that do not vanish identically on $T_{(X,\omega)}(\cM)\cap\ker(p)$.
\end{rem}
The local description of rank 1 affine  invariant submanifold in terms of the flat geometry of surfaces is crucial for our approach. In particular none of our arguments are applicable in higher rank.

\subsection{Holonomy field, trace field and (affine) field of definition}
For more details on the results in this section, see~\cite{KS,Wfield}.
\par
The {\em holonomy field} of a translation surface is the smallest subfield of $\bR$ such that the absolute periods are a vector space of dimension two over this field \cite{KS}.  The {\em trace field} of a translation surface is the subfield of $\mathbb R$ generated by the traces of elements of its Veech group. This is also $\bQ[\lambda+\lambda^{-1}]$ where $\lambda$ is the largest eigenvalue of {\em any} hyperbolic element of $SL(X,\omega)$. When $SL(X,\omega)$ contains a hyperbolic, the trace field and the holonomy field coincide.
\par
The (affine) field of definition of an affine invariant submanifold $\cM$ (introduced in~\cite{Wfield}), denoted $\bk(\cM)$, is defined to be the smallest subfield of $\bR$ such that locally $\cM$ can be described by homogeneous linear equations in period coordinates with coefficients in $\bk(\cM)$. It is a number field~\cite{Wfield} and is always totally real \cite{Fi1}.
\par
\begin{lem}\label{L:sameholfield}
If $\cM$ is a rank 1 affine invariant submanifold then any translation surface $(X,\omega)\in \cM$ has holonomy field equal to $\bk(\cM)$.
\end{lem}
\par
\begin{proof}
By \cite[Theorem 1.1]{Wfield}, $\bk(\cM)$ is the intersection of the holonomy fields of all translation surfaces in $\cM$. By definition, the holonomy field is invariant under both the $GL(2, \bR)$-action and rel deformations.

Since  $\cM$ is rank 1, any two points in $\cM$ may be connected by using the $\mathrm{GL}(2, \bR)$-action and rel deformations. (Indeed, by Proposition \ref{P:SL2plusrel} any path in $\cM$ between nearby points  can be isotoped to a path in $\cM$ which first follows a $GL(2, \bR)$ orbit and then follows $\ker(p)$. Any path in $\cM$ can be broken up into a suite of paths connecting nearby points.) This proves the lemma. 
\end{proof}
\par
\begin{defn}
We will say that $\cM$ is arithmetic if $\bk(\cM)=\bQ$.
\end{defn}
\par

\begin{rem}
In the rank 1 case, it is not hard to show that this definition of arithmetic coincides with the one given in the introduction. 
\end{rem}
\par
\begin{cor}\label{C:same}
If $\cM'\subset \cM$ are rank 1, then $\bk(\cM)=\bk(\cM')$. In particular $\cM$ is arithmetic if and only if
$\cM'$ is arithmetic.
\end{cor}

\begin{lem}\label{L:relperiods}
If $(X, \omega)\in \cM$ is a completely parabolic surface, then there exists $g\in GL(2, \bR)$, arbitrarily close to the identity, such that $g(X,\omega)$ has absolute and relative periods in $\bk[i]$, where $\bk = \bk(\cM)$.
\end{lem}

\begin{proof}
This follows from the Thurston-Veech construction \cite{Th, V}. (This uses only that $SL(X,\omega)$ contains two non-commuting parabolic elements.)

Alternatively, by the definition of holonomy field there is  $g\in GL(2, \bR)$ arbitrarily close to the identity such that $g(X,\omega)$ has absolute periods in $\bk[i]$. By \cite[Theorem 30]{KS} or \cite[Theorem 9.4]{Mc6}, if $SL(X, \omega)$ has a hyperbolic element, then the relative and absolute periods span the same $\bQ$ vector subspace of $\bC$. (This uses only that $SL(X,\omega)$ contains a hyperbolic element.)
\end{proof}
\par
Suppose that $\bk$ is a totally real subfield of $\bR$ of degree $d$ over $\bQ$. We denote the field embeddings of $\bk$ into $\bR$ by $\iota_1, \ldots, \iota_d$, where $\iota_1$ is the identity embedding.

Consider a subgroup $\Gamma \subset SL(2,\bk)$. By applying $\iota_i$ to each matrix entry, we obtain a homomorphism $\rho_i:\Gamma \to SL(2,\bR)$ satisfying $tr(\rho_i(g))=\iota_i(tr(g))$.
Define $\rho:\Gamma \to SL(2, \bR)^d$ by $\rho(g)=(\rho_1(g), \ldots, \rho_d(g))$. 
\par
The next proposition is a consequence of~\cite[Proposition 2.1 and Corollary 2.2]{Geninska}.

\begin{prop}\label{P:Zdense}
Let $\Gamma \subset SL(2,\bk)$ be a non-elementary subgroup having a totally real trace field $\bk$ of degree $d$ over $\bQ$. If there exists $g\in \Gamma$ with eigenvalue $\lambda$ such that $\bk=\bQ[\lambda^2+\lambda^{-2}]$, then the image $\rho(\Gamma)$ of $\Gamma$ in $SL(2,\bR)^d$ is Zariski dense. 
\end{prop}

\begin{proof}
We provide a sketch for convenience. 

Let $\mathfrak{g}$ be the complexification of the Lie algebra of the Zariski closure of $\rho(\Gamma)\subset SL(2, \bR)^d$. It suffices  to show $\mathfrak{g}=\bigoplus_{i=1}^d \mathfrak{sl}(2, \bC)$. 

Step 1: Let $\pi_i: \bigoplus_{i=1}^d \mathfrak{sl}(2, \bC)\to \mathfrak{sl}(2, \bC)$ be the projection onto the $i$-th factor.  For each $i$, we have that $\pi_i(\mathfrak{g})=\mathfrak{sl}(2, \bC)$. This follows since the Adjoint action of $\rho(\Gamma)$ on each coordinate of $\bigoplus_{i=1}^d \mathfrak{sl}(2, \bC)$ has no invariant subspaces, and $\mathfrak{g}$ is an invariant subspace for the Adjoint action of $\rho(\Gamma)$ on $\bigoplus_{i=1}^d \mathfrak{sl}(2, \bC)$.  

Step 2: For each $i$, we have that $\mathfrak{g}$ contains an element which is non-zero only in coordinate $i$. This follows from the previous result, and the following observations. $\Ad(g)$ has eigenvalues $\lambda^2, \lambda^{-2}$ and 1.  Thus if $g\in \Gamma$ has an eigenvalue $\lambda$ with $\bk=\bQ[\lambda^2+\lambda^{-2}]$, then $\Ad(\rho(g))$ has a simple eigenvector that is non-zero only in coordinate $i$. Basic linear algebra gives that the projection onto the line spanned by this eigenvector can be written as a polynomial in $\Ad(\rho(g))$. 

Step 3: Again using that $\Ad(\rho(\Gamma))$ acting on each $\mathfrak{sl}(2, \bC)$ coordinate has no invariant subspaces, we get the result. 
\end{proof}

\begin{rem}
By the result of Kenyon-Smillie mentioned above, when $SL(X,\omega)$ contains a hyperbolic element then $SL(X,\omega)$ can be conjugated into $SL(2, \bk)$, where $\bk$ is the holonomy field of $(X,\omega)$. Furthermore, if $g\in SL(X, \omega)$ is hyperbolic with eigenvalues $\lambda>1$ and $\lambda^{-1}$, then since $g^2$ is also hyperbolic $\bk=\bQ[tr(g^2)]=\bQ[\lambda^2+\lambda^{-2}]$. Hence Proposition~\ref{P:Zdense} is applicable to non-elementary Veech groups. 
\end{rem}

\section{Proof of Theorem \ref{T:mother}}\label{S:proof}

For the duration of this section, in order to find a contradiction, we assume that $\cM$ is a non-arithmetic rank 1 affine invariant submanifold,  that completely parabolic surfaces are dense in $\cM$, and that $\cM$ is not a closed orbit. Let $\bk=\bk(\cM)$ be the  (affine) field of definition of $\cM$. The definition of non-arithmetic means that $\bk \neq \bQ$.
We denote the field embeddings of $\bk$ into $\bR$ by $\iota_1, \ldots, \iota_d$.
\begin{lem}\label{lem:exist:CP:stab}
There exists a completely parabolic surface $(X,\omega)\in \cM$ that is $\cM$-stably periodic in the horizontal direction, has relative periods in $\bk[\imath]$, and satisfies the following property. There exist a pair of horizontal cylinders $C_1, C_2$ on $(X,\omega)$ such that the ratio of the  moduli of $C_1$ and $C_2$ is not  constant under small rel deformations in $\cM$. 
\end{lem}
\begin{proof}
Given any periodic direction on any surface in $\cM$, a generic small deformation of this surface will be $\cM$-stably periodic in the same direction. Thus there is a surface $(X', \omega')\in \cM$ with an $\cM$-stably periodic direction. There is small neighborhood $U$ of $(X', \omega')$ where the cylinders in this direction persist, and the direction stays $\cM$-stably periodic.
Since we have assumed that completely parabolic surfaces are dense in $\cM$, there must be a completely parabolic surface $(X, \omega)$ in $U$. After applying an element in $GL(2, \bR)$, we may assume that the $\cM$-stably periodic direction on $(X, \omega)$ is horizontal and that the absolute and hence also relative periods lie in $\bk[i]$; compare to the proof of Lemma \ref{L:relperiods}.
\par
By Proposition \ref{P:SL2plusrel}, there are rel deformations of $(X,\omega)$ which remain in $\cM$.
For any non-trivial imaginary rel deformation of an $\cM$-stably horizontally periodic surface, there is a cylinder that increases in height, and another that decreases in height. Two such cylinders can be chosen as $C_1$ and $C_2$. Since rel deformations do not change circumference, the ratio of moduli of $C_1$ and $C_2$ changes under the imaginary rel deformation.
\end{proof}
Let $(X,\omega)\in \cM$, and $C_1,C_2\subset X$  be as in Lemma~\ref{lem:exist:CP:stab}. For $v$ in a neighborhood of $0$ in $\ker(p)\cap T_{(X,\omega)}(\cM)$, define $f(v)$ to be the ratio of the moduli of $C_1$ and $C_2$. Note that $f$ is the quotient of two degree 1 polynomials  in the {\em imaginary coordinates} of $v$ with coefficients in $\bk(\cM)$. More precisely, if  the heights and circumferences of $C_i$ at $(X, \omega)$ are denoted by  $h_i$ and $c_i$ respectively, then there are   linear functionals $\alpha_1,\alpha_2$ with coefficients in $\bk$ such that
$$
f(v) = \frac{c_2}{c_1}\cdot\frac{h_1+\alpha_1(\mathrm{Im}(v))}{h_2+\alpha_2(\mathrm{Im}(v))}.
$$
Using this formula as a definition, we extend $f$ to all of $\ker(p)\cap T_{(X,\omega)}(\cM)$. 
\begin{rem}
For large $v$, $f(v)$ may not have geometric meaning. There are several reasons for this: the cylinders $C_1$ and $C_2$ may not persist on $(X, \omega)+v$, or worse yet, $(X, \omega)+v$ may not even be well defined (zeroes may collide). Furthermore, if a path $(X, \omega)+tv$ ceases to be $\cM$-stably periodic at some $t=t_0$, then $f(tv)$ may not represent the ratio of the moduli of $C_1$ and $C_2$ for $t>t_0$, even if $C_1$ and $C_2$ persist. This is because some zero can ``hit" the boundary of the $C_i$. (The heights of cylinders are only piecewise linear functions.)
\end{rem}
\par
Define
$$
P = \left\{ v\in \ker(p)\cap T_{(X,\omega)}(\cM):f(tv)=f(0)\textrm{ for all } t\in \bR \right\}.
$$
Observe that $P$ is the kernel of $\alpha_1 h_2-\alpha_2h_1$, so it is a real hyperplane of $\ker(p)\cap T_{(X,\omega)}(\cM)$. 
\par
\begin{lem}\label{L:polys}
Let $v\in \ker(p)\cap T_{(X,\omega)}(\cM) \cap H^1(X,\Sigma, \bk[i])$ be a vector satisfying $v\notin P$. There exists a collection of $d-1$ polynomials (depending on $v$) of degree at most 2 in the $4d$ entries of $\rho(g)$ such that for any $g\in \mathrm{SL}(2,\bk)$, if any of these polynomials are non-zero at $\rho(g)$, then $f(gv)\not \in\bQ$.  Moreover, these polynomials do not vanish identically on $\rho(SL(2,\bk))$.
\end{lem}
\begin{proof}
The condition that $f(gv)\in \bQ$ is equivalent to the $d-1$ equations
$$\iota_i(f(gv)) = \iota_{i+1}(f(gv)), \quad i=1, \ldots, d-1.$$
Since $f$ is the ratio of two degree one polynomials, by clearing denominators in each of these equations we get a system of $d-1$ quadratic equations in the coefficients of $\rho(g)$ that vanish whenever $f(gv)\in \bQ$.
\par
It remains to show that these polynomials do not vanish identically on $\rho(SL(2,\bk))$. For this, consider the matrices
$$
a_t=\left(\begin{array}{cc} t^{-1} & 0\\0 &t\end{array}\right)
$$
for $t\in \bk^*$. Since $f(v)$ depends only on the imaginary parts of $v$, we get that for any $t\in \bk^*$, $f(a_t v) = f(tv)$. Since $v\notin P$ the function $t\mapsto f(tv)$ is not constant. Recall that $\bk\neq \bQ$, hence any non-constant function $\bk \to \bR$ defined as the ratio of two degree one polynomials cannot always take rational values. This proves the lemma.
\end{proof}
\par
For any $v\in \ker(p)\cap T_{(X,\omega)}(\cM) \cap H^1(X,\Sigma, \bk[i])$ with $v\notin P$, we will denote the algebraic subvariety of $SL(2,\bR)^d$ defined by the polynomials in Lemma~\ref{L:polys} by $\cX(v)$. In particular $\rho(SL(2,\bk))  \not \subset \cX(v)$.
 \par
\begin{lem}\label{L:findg}
There exists $R>0$ such for any $v\in \ker(p)\cap T_{(X,\omega)}(\cM)\cap H^1(X,\Sigma, \bk[i])$ with $v\notin P$, there exists $g\in SL(X,\omega)$ such that $\|g\|\leq R$ and $\rho(g) \not \in \cX(v)$ {\em i.e.} $f(gv)\notin \bQ$.
\end{lem}
\par

The non-trivial part of Lemma~\ref{L:findg} is the control of the size of  $g$. This is done with the following result of Eskin-Mozes-Oh. The statement presented here is slightly stronger than~\cite[Proposition 3.2]{EMO} (see~\cite[Lemma 2.5]{BG} where the statement below appears).

\par
\begin{thm}[Eskin-Mozes-Oh]\label{T:EMO}
For any $k>0$, there is an $N>0$ such that for any integer $m\geq 1$, any field $\bk$, any 
algebraic subvariety $\cX\subset \mathrm{GL}(m,\bk)$ defined by at most $k$ polynomials of degree at most $k$, and any family $S\subset \mathrm{GL}(m,\bk)$ that generates a subgroup which is not contained in $\mathcal X$, we have $S^N \not \subseteq \mathcal X$.
\end{thm}
\par
Here $S^N$ denotes the word ball of radius $N$.
\par
\begin{proof}[Proof of Lemma~\ref{L:findg}]
Let $S$ be a set of two non-commuting parabolic elements in $SL(X, \omega)$, and let $\Gamma$ be the subgroup generated by $S$. Since the relative periods of $(X,\omega)$ are in $\bk[i]$ we have $SL(X, \omega)\subset SL(2,\bk)$. By Proposition \ref{P:Zdense}, $\rho(\Gamma)$ is Zariski dense in $SL(2,\bR)^d$. 

Let $v$ be as in the statement of the lemma. By Lemma \ref{L:polys} we know that $\rho(SL(2,\bk))$  is not contained in the subvariety $\mathcal {X}(v)$ of $SL(2,\bR)^d\subset GL(2d,\bR)$. Thus the subgroup $\Gamma$ is not contained in this variety. Theorem \ref{T:EMO} then implies that there is an $N$ independent of $v$ such that there exists an element $g \in SL(X,\omega)$ of word length in $S$  at most $N$, such that $\rho(g)$ does not belong to $\mathcal{X}(v)$, which means that $f(gv)\notin \bQ$ by Lemma~\ref{L:polys}.

If $M=\max_{h\in S} \|h\|$, we may set $R=M^N$, and hence $\|g\|\leq R$.
\end{proof}
\begin{cor}\label{cor:v:small:no:cp}
For all sufficiently small elements $v $ in $ \ker(p)\cap T_{(X,\omega)}(\cM)\cap H^1(X,\Sigma,\bk[i])$, if $v\not\in P$ then $(X,\omega)+v$ is not completely parabolic.
\end{cor}
\begin{proof}
Let  $R$ be the constant given by Lemma~\ref{L:findg}. Let $U$ be the neighborhood of $0$ in $\ker(p)\cap T_{(X,\omega)}(\cM)$ given by Lemma \ref{L:gonrel} such that
the map $v\mapsto (X,\omega)+v$ is well defined and injective, and such that if $g\in SL(X,\omega)$ and $v\in U$ then $g((X,\omega)+v)=(X,\omega)+gv$. By making $U$ smaller if necessary, we can assume that the  cylinders $C_1$ and $C_2$ persist at $(X,\omega)+v$ for all $v\in U$ and that $f(v)$ gives { the ratio} of the moduli of $C_1$ and $C_2$ at $(X, \omega)+v$. 

Now consider any $v$ as above that is small enough so that $w\in U$ for all $w$ with $\|w\|\leq R\|v\|$.
If $v\not\in P$, Lemma~\ref{L:findg} furnishes an element $g\in SL(X,\omega)$ such that $f(gv)\notin \bQ$ and $\|g\|\leq R$. But $g((X,\omega)+v)=(X,\omega)+gv$.  Since the function $f$ is the ratio of the moduli of $C_1$ and $C_2$,  we see that $(X,\omega)+gv$ has a pair of parallel  cylinders whose ratio of moduli is not rational. Since $(X,\omega)+gv$ belongs to the $SL(2,\bR)$ orbit of $(X,\omega)+v$, we see that $(X, \omega)+v$ is not completely parabolic.
\end{proof}

\begin{proof}[Conclusion of the proof of Theorem \ref{T:mother}]
Recall that the set of completely parabolic surfaces is $GL(2, \bR)$ invariant. Since $\cM$ is rank 1,  if completely parabolic surfaces are dense, then there is a neighborhood $U$ of $0$ in $\ker(p)\cap T_{(X,\omega)}(\cM)$ such that  $(X,\omega)+v$ is completely parabolic for a dense set of $v\in U$.

Since $(X, \omega)$ has  relative periods in $\bk[i]$, if $(X,\omega)+v$ is completely parabolic then it must have relative periods in $\bk[i]$, and hence $v$ has coordinates in $\bk[i]$. By Corollary~\ref{cor:v:small:no:cp}, if $v\notin P$ is small enough and has coordinates in $\bk[i]$, then $(X, \omega)+v$ is not completely parabolic. Since $P$ is a hyperplane, this contradicts the fact that the set of completely parabolic surfaces is dense in $\cM$. Theorem \ref{T:mother} is proved.
\end{proof}



\bibliography{mybib}{}

\providecommand{\bysame}{\leavevmode\hbox to3em{\hrulefill}\thinspace}
\providecommand{\MR}{\relax\ifhmode\unskip\space\fi MR }
\providecommand{\MRhref}[2]{%
  \href{http://www.ams.org/mathscinet-getitem?mr=#1}{#2}
}
\providecommand{\href}[2]{#2}
\begin{thebibliography}{McM06b}

\bibitem[AEM]{AEM}
Artur Avila, Alex Eskin, and Martin M\"oller, \emph{Symplectic and {I}sometric
  {$SL(2, \bR)$}-invariant subbundles of the {H}odge bundle}, preprint, arXiv
  1209.2854 (2012).

\bibitem[ANW]{ANW}
David Aulicino, Duc-Manh Nguyen, and Alex Wright, \emph{Classification of
  higher rank orbit closures in $\mathcal{H}^{\rm odd}(4)$}, preprint, arXiv
  1308.5879 (2013), to appear in Journal of the Europ. Math. Soc.

\bibitem[BG08]{BG}
E.~Breuillard and T.~Gelander, \emph{Uniform independence in linear groups},
  Invent. Math. \textbf{173} (2008), no.~2, 225--263.

\bibitem[BHM]{BHM}
Matt Bainbridge, Philipp Habegger, and Martin Moeller, \emph{Teichmueller
  curves in genus three and just likely intersections in $g_m^n x g_a^n$},
  preprint, arXiv 1410.6835 (2014).

\bibitem[BM12]{BaM}
Matt Bainbridge and Martin M{\"o}ller, \emph{The {D}eligne-{M}umford
  compactification of the real multiplication locus and {T}eichm\"uller curves
  in genus 3}, Acta Math. \textbf{208} (2012), no.~1, 1--92.

\bibitem[Cal04]{Ca}
Kariane Calta, \emph{Veech surfaces and complete periodicity in genus two}, J.
  Amer. Math. Soc. \textbf{17} (2004), no.~4, 871--908.

\bibitem[EM]{EM}
Alex Eskin and Maryam Mirzakhani, \emph{Invariant and stationary measures for
  the {$SL(2,\bR)$} action on moduli space}, preprint, arXiv 1302.3320 (2013).

\bibitem[EMM15]{EMM}
Alex Eskin, Maryam Mirzakhani, and Amir Mohammadi, \emph{Isolation,
  equidistribution, and orbit closures for the {$SL(2, \bR)$}-action on moduli
  space}, Annals of Math. \textbf{182} (2015), no.~2, 673--721.

\bibitem[EMO05]{EMO}
Alex Eskin, Shahar Mozes, and Hee Oh, \emph{On uniform exponential growth for
  linear groups}, Invent. Math. \textbf{160} (2005), no.~1, 1--30.

\bibitem[Fil16]{Fi1}
Simion Filip, \emph{Splitting mixed {H}odge structures over affine invariant
  manifolds}, Annals of Math. \textbf{183} (2016), no.~2, 681--713.

\bibitem[FM12]{FarMar}
Benson Farb and Dan Margalit, \emph{A primer on mapping class groups},
  Princeton Mathematical Series, vol.~49, Princeton University Press,
  Princeton, NJ, 2012.

\bibitem[Gen12]{Geninska}
Slavyana Geninska, \emph{Examples of infinite covolume subgroups of
  $\mathrm{PSL}(2,\mathrm{R})^r$ with big limit sets}, Math Z \textbf{272}
  (2012), no.~1-2, 389--404.

\bibitem[GHS03]{GHS}
Eugene Gutkin, Pascal Hubert, and Thomas~A. Schmidt, \emph{Affine
  diffeomorphisms of translation surfaces: periodic points, {F}uchsian groups,
  and arithmeticity}, Ann. Sci. \'Ecole Norm. Sup. (4) \textbf{36} (2003),
  no.~6, 847--866 (2004).

\bibitem[HS04]{HS2}
Pascal Hubert and Thomas~A. Schmidt, \emph{Infinitely generated {V}eech
  groups}, Duke Math. J. \textbf{123} (2004), no.~1, 49--69.

\bibitem[Kon97]{K}
M.~Kontsevich, \emph{Lyapunov exponents and {H}odge theory}, The mathematical
  beauty of physics ({S}aclay, 1996), Adv. Ser. Math. Phys., vol.~24, World
  Sci. Publ., River Edge, NJ, 1997, pp.~318--332.

\bibitem[KS00]{KS}
Richard Kenyon and John Smillie, \emph{Billiards on rational-angled triangles},
  Comment. Math. Helv. \textbf{75} (2000), no.~1, 65--108.

\bibitem[LM]{ML}
Erwan Lanneau and Martin M{\"o}ller, \emph{Finiteness results for teichm\"uller
  curves in prym loci}, preprint (2014).

\bibitem[LN]{LN2}
Erwan Lanneau and Duc-Manh Nguyen, \emph{$\mathrm{GL}^+(2,\mathbb{R})$-orbits
  in {P}rym eigenform loci}, preprint, arXiv 1310.8537 (2014), to appear in
  Geom. and Topol.

\bibitem[LN16]{LN}
\bysame, \emph{Complete {P}eriodicity of {P}rym {E}igenforms}, Ann. Sci. de
  l'E.N.S. \textbf{49} (2016), no.~1, 87--130.

\bibitem[McM03a]{Mc}
Curtis~T. McMullen, \emph{Billiards and {T}eichm\"uller curves on {H}ilbert
  modular surfaces}, J. Amer. Math. Soc. \textbf{16} (2003), no.~4, 857--885
  (electronic).

\bibitem[McM03b]{Mc6}
\bysame, \emph{Teichm\"uller geodesics of infinite complexity}, Acta Math.
  \textbf{191} (2003), no.~2, 191--223.

\bibitem[McM05a]{McM:spin}
\bysame, \emph{Teichm\"uller curves in genus two: discriminant and spin}, Math.
  Ann. \textbf{333} (2005), no.~1, 87--130.

\bibitem[McM05b]{Mc7}
\bysame, \emph{Teichm\"uller curves in genus two: the decagon and beyond}, J.
  Reine Angew. Math. \textbf{582} (2005), 173--199.

\bibitem[McM06a]{Mc2}
\bysame, \emph{Prym varieties and {T}eichm\"uller curves}, Duke Math. J.
  \textbf{133} (2006), no.~3, 569--590.

\bibitem[McM06b]{Mc4}
\bysame, \emph{Teichm\"uller curves in genus two: torsion divisors and ratios
  of sines}, Invent. Math. \textbf{165} (2006), no.~3, 651--672.

\bibitem[McM07]{Mc5}
\bysame, \emph{Dynamics of {${\rm SL}_2(\Bbb R)$} over moduli space in genus
  two}, Ann. of Math. (2) \textbf{165} (2007), no.~2, 397--456.

\bibitem[McM09]{Mc3}
\bysame, \emph{Rigidity of {T}eichm\"uller curves}, Math. Res. Lett.
  \textbf{16} (2009), no.~4, 647--649.

\bibitem[M{\"o}l06a]{M2}
Martin M{\"o}ller, \emph{Periodic points on {V}eech surfaces and the
  {M}ordell-{W}eil group over a {T}eichm\"uller curve}, Invent. Math.
  \textbf{165} (2006), no.~3, 633--649.

\bibitem[M{\"o}l06b]{M}
\bysame, \emph{Variations of {H}odge structures of a {T}eichm\"uller curve}, J.
  Amer. Math. Soc. \textbf{19} (2006), no.~2, 327--344 (electronic).

\bibitem[M{\"o}l08]{M3}
\bysame, \emph{Finiteness results for {T}eichm\"uller curves}, Ann. Inst.
  Fourier (Grenoble) \textbf{58} (2008), no.~1, 63--83.

\bibitem[MT02]{MT}
Howard Masur and Serge Tabachnikov, \emph{Rational billiards and flat
  structures}, Handbook of dynamical systems, {V}ol.\ 1{A}, North-Holland,
  Amsterdam, 2002, pp.~1015--1089.

\bibitem[MW15]{MW}
Carlos Matheus and Alex Wright, \emph{Hodge-{T}eichm{\"u}ller planes and
  finiteness results for {T}eichm{\"u}ller curves}, Duke Math. J. \textbf{164}
  (2015), no.~6, 1041--1077.

\bibitem[NW14]{NW}
Duc-Manh Nguyen and Alex Wright, \emph{Non-{V}eech surfaces in
  $\mathcal{H}^{\rm hyp}(4)$ are generic}, Geom. Funct. Anal. \textbf{24}
  (2014), no.~4, 1316--1335.

\bibitem[SW07]{SWprobs}
John Smillie and Barak Weiss, \emph{Finiteness results for flat surfaces: a
  survey and problem list}, Partially hyperbolic dynamics, laminations, and
  {T}eichm\"uller flow, Fields Inst. Commun., vol.~51, Amer. Math. Soc.,
  Providence, RI, 2007, pp.~125--137.

\bibitem[SW08]{SW3}
\bysame, \emph{Veech's dichotomy and the lattice property}, Ergodic Theory
  Dynam. Systems \textbf{28} (2008), no.~6, 1959--1972.

\bibitem[Thu88]{Th}
William~P. Thurston, \emph{On the geometry and dynamics of diffeomorphisms of
  surfaces}, Bull. Amer. Math. Soc. (N.S.) \textbf{19} (1988), no.~2, 417--431.

\bibitem[Vee89]{V}
W.~A. Veech, \emph{Teichm\"uller curves in moduli space, {E}isenstein series
  and an application to triangular billiards}, Invent. Math. \textbf{97}
  (1989), no.~3, 553--583.

\bibitem[Wri14]{Wfield}
Alex Wright, \emph{The field of definition of affine invariant submanifolds of
  the moduli space of abelian differentials}, Geom. Top. \textbf{18} (2014),
  no.~3, 1323--1341.

\bibitem[Wri15a]{Wcyl}
\bysame, \emph{Cylinder deformations in orbit closures of translation
  surfaces}, Geom. Top. \textbf{19} (2015), no.~1, 413--438.

\bibitem[Wri15b]{Wsurvey}
\bysame, \emph{Translation surfaces and their orbit closures: An introduction
  for a broad audience}, EMS Surv. Math. Sci. \textbf{2} (2015), no.~1,
  63--108.

\bibitem[Zor06]{Z}
Anton Zorich, \emph{Flat surfaces}, Frontiers in number theory, physics, and
  geometry. {I}, Springer, Berlin, 2006, pp.~437--583.

\end{thebibliography}
\bibliographystyle{amsalpha}
\end{document}